\documentclass[reqno]{amsart}
\usepackage{amsmath}
\usepackage{amsfonts}
\usepackage{amssymb}
\usepackage{txfonts}
\usepackage[T1]{fontenc}
\newtheorem{theorem}{Theorem}
\newtheorem{corollary}{Corollary}
\newtheorem{definition}{Definition}
\newtheorem{example}{Example}
\newtheorem{lemma}{Lemma}

\newtheorem{remark}{Remark}

\DeclareMathOperator{\dom}{dom}
\DeclareMathOperator{\coker}{coker}
\DeclareMathOperator{\ev}{ev}
\DeclareMathOperator{\im}{Im}

\DeclareMathOperator{\fix}{Fix}
\DeclareMathOperator{\co}{co}
\DeclareMathOperator{\ind}{Ind}

\newcommand{\map}{\multimap}
\newcommand{\<}{\leqslant}

\newcommand{\R}[1]{\varmathbb{R}^{#1}}

\begin{document}
\title[Guiding potential method for differential inclusions with nonlocal conditions]{Guiding potential method for differential inclusions with nonlocal conditions}
\author{Rados\l aw Pietkun}
\address{Toru\'n, Poland}
\email{rpietkun@pm.me}

\begin{abstract}
The existence of solutions of some nonlocal initial value problems for differential inclusions is established. The guiding potential method is used and the topological degree theory for admissible multivalued vector fields is applied. Some conclusions concerning compactness of the solution set have been drawn. 
\end{abstract}
\maketitle

\section{Introduction}
\par The aim of this paper is to formulate theorems about the existence of absolutely continuous solutions of the following nonlocal Cauchy problem
\begin{equation}\label{nonlocal}
\begin{cases}
\dot{x}(t)\in F(t,x(t)),&\mbox{a.e. }t\in I=[0,T],\\
x(0)=g(x),
\end{cases}
\end{equation}
where $F\colon I\times\R{N}\map\R{N}$ is a multivalued map and $g\colon C(I,\R{N})\to\R{N}$. Proofs of these theorems rely on the use of $C^1$-guiding potential. This method, whose base was laid by M. A. Krasnosel'skii and A. I. Perov, generically employed to differential equations and subsequently expanded to inclusions, has demonstrated its effectiveness to the study of periodic problems (see, e.g. \cite{blasi,gorn,zabreiko} and especially \cite{obu}, which contains a lot of references to subject literature). However, periodic condition constitutes just a single case of a much wider class of so called nonlocal initial conditions. The consideration for nonlocal initial condition, given by some $g$, is stimulated by the observation that this type of conditions is more realistic than usual ones in treating physical problems. \par Let us mention here three important cases covered by our results:
\begin{itemize}
\item $g(x)=-x(T)$ (anti-periodicity condition);
\item $g(x)=\sum_{i=1}^n\alpha_ix(t_i)$, where $\sum_{i=1}^n|\alpha_i|\<1$, $\sum_{i=1}^n\alpha_i\neq 1$ and $0<t_1<t_2<\ldots<t_n\<T$ (multi-point discrete mean condition);
\item $g(x)=\frac{1}{T}\int_0^Th(x(t))\,dt$, with $h\colon\R{N}\to\R{N}$ such that $|h(x)|\<|x|$ (mean value condition).
\end{itemize}
\par In each of the above cases we are dealing with the situation when the function $g$ possesses a sublinear growth, and strictly speaking satisfies the condition: $|g(x)|\<||x||$ for $x\in C(I,\R{N})$. This means in particular that the existence of solutions to the problem \eqref{nonlocal} for such g's can not be resolved by means of the fixed point theorem (see Theorem \ref{th9}. below). Therefore arises a need for a different approach to boundary problems of this type. Such an approach is presented in Theorems \ref{th1}., \ref{th2}. and \ref{th3}., which exploit the existence of a smooth guiding potential $V$ for the multivalued right-hand side of the differential inclusion $\dot{x}(t)\in F(t,x(t))$. Next, we try to abstract properties of the map appearing in the nonlocal condition by formulating Theorem \ref{th4}., which generalizes the previous considerations. \par In Theorem \ref{th6}. we give an example of the property of $g$, which also guarantees the existence of solutions to the problem \eqref{nonlocal}, but that excludes same time sublinear growth \eqref{sublinear} with a constant $c\<1$.\par We end this paper by describing the topological structure of the set of solutions to \eqref{nonlocal} in Theorem \ref{compact}., which gives conditions for the compactness of this set.

\section{Preliminaries}
Let $X$ and $Z$ be a Banach space. An open ball with center $x\in X$ (resp. zero) and radius $r>0$ is denoted by $B(x,r)$ ($B(r)$). Symbol $\overline{B}(r)$ stands for a closed ball. If $A\subset X$, then $\overline{A}$ denotes the closure of $A$, $\partial A$ the boundary of $A$ and $\co A$ the convex hull of $A$. The inner product in $\R{N}$ represents $\langle\cdot,\cdot\rangle$. 
\par For $I\subset\R{}$, $(C(I,\R{N}),||\cdot||)$ is the Banach space of continuous maps $I\to \R{N}$ equipped with the maximum norm and $AC(I,\R{N})$ is the subspace of absolutely continuous functions. By $(L^1(I,\R{N}),||\cdot||_1)$ we mean the Banach space of all Lebesgue integrable maps.
\par A multivalued map $F\colon X\map Z$ assigns to any $x \in X$ a nonempty subset $F(x)\subset Z$. The set of all fixed points of the multivalued (or univalent) map $F$ is denoted by $\fix(F)$. Recall that a map $F\colon X\map Z$ with compact values is upper semicontinuous iff given a sequence $(x_n,y_n)$ in the graph of $F$ with $x_n\to x$ in $X$, there is a subsequence $y_{k_n}\to y\in F(x)$. If the image $F(X)$ is relatively compact in $Z$, then we say that $F$ is a compact multivalued map. A multimap $F\colon I\map\R{N}$ is called measurable, if $\{t\in I\colon F(t)\subset A\}$ belongs to the Lebesgue $\sigma$-field of $I$ for every closed $A\subset\R{N}$. 
\par A set-valued map $F\colon X\map Z$ is admissible (compare \cite[Def.40.1]{gorn2}) if there is a metric space $Y$ and two continuous functions $p\colon X\to Y$, $q\colon Y\to Z$ from which $p$ is a Vietoris map such that $F(x)=q(p^{-1}(x))$ for every $x\in X$. It turns out that every acyclic multivalued map, i.e. an usc multimap with compact acyclic values, is admissible. In particular, every usc multivalued map with compact convex values is admissible.\par Let ${\mathcal M}$ be the set of triples $(Id-F,\Omega,y)$ such that $\Omega\subset X$ is open bounded, $Id$ is the identity, $F\colon\overline{\Omega}\map X$ is a compact usc multimap with closed convex values, and $y\not\in(Id-F)(\partial\Omega)$. Then it is possible to define, using approximation methods for multivalued maps, a unique topological degree function $\deg\colon{\mathcal M}\to\varmathbb{Z}$ (see \cite{deimling,gorn2,obu} for details). This degree inherits directly all the basic properties of the Leray-Schauder degree, among others existence, localization, normality, additivity, homotopy invariance and contractivity.\par Let $L\colon\dom L\subset X\to Z$ be a linear Fredholm operator of zero index such that $\im L\subset Z$ is a closed subspace. Consider continuous linear idempotent projections $P\colon X\to X$ and $Q\colon Z\to Z$ such that $\im P=\ker L$, $\im L=\ker Q$. By the symbol $L_P$ we denote the restriction of the operator $L$ on $\dom L\cap\ker P$ and by $K_{P,Q}\colon Z\to\dom L\cap\ker P$ the operator given by the relation $K_{P,Q}(z)=L_P^{-1}(Id-Q)(z)$. Since $\dim\ker L=\dim\coker L<\infty$ we may choose a linear homeomorphism $\Phi\colon\im Q\to\im P$.
\par The proofs of main results of this paper will be based in particular on the following coincidence point theorem. In the wake of \cite{maw} we should specify it as the Continuation Theorem.
\begin{theorem}\label{conth}\cite[Lemma 13.1.]{deimling} Let $X$ and $Z$ be real Banach spaces, $L\colon\dom L\subset X\to Z$ be Fredholm operator of index zero and with closed graph, $\Omega\subset X$ be open bounded and $N\colon\overline{\Omega}\map Z$ be such that  $QN$ and $K_{P,Q}N$ are compact usc multimaps with compact convex values. Assume also that
\begin{itemize}
\item[(a)] $Lx\not\in\lambda N(x)$ for all $\lambda\in(0,1)$ and $x\in\dom L\cap\partial\Omega$,
\item[(b)] $0\not\in QN(x)$ on $\ker L\cap\partial\Omega$ and $\deg(\Phi QN,\ker L\cap\Omega,0)\neq 0$.
\end{itemize}   
Then $\deg(Id-P-(\Phi Q+K_{P,Q})N,\Omega,0)\neq 0$. In particular, $Lx\in N(x)$ has a solution in $\Omega$.
\end{theorem}
We are able to place considered boundary value problem \eqref{nonlocal} in the context of Theorem \ref{conth}. and the above introduced general framework. Let $X:=C(I,\R{N})$, $Z:=L^1(I,\R{N})\times\R{N}$, $\dom L:=AC(I,\R{N})$; $L\colon\dom L\subset X\to Z$ be such that $Lx:=(\dot{x},0)$. Then $\ker L=i(\R{N})$, where $i\colon\R{N}\hookrightarrow C(I,\R{N})$ is defined by $i(x_0)(t)=x_0$, $\im L=L^1(I,\R{N})\times \{0\}$ and $\coker L\approx\R{N}$, i.e. $L$ is a Fredholm mapping of index zero. Consider continuous linear operators $P\colon X\to X$ and $Q\colon Z\to Z$ such that $P(x)(t)=x(0)$ and $Q((y,v))=(0,v)$. It is clear that $(P,Q)$ is an exact pair of idempotent projections with respect to $L$. Define operator $N_F\colon X\to Z$ by \[N_F(x):=\left\{f\in L^1(I,\R{N})\colon f(t)\in F(t,x(t))\mbox{ for a.a. }t\in I\right\}\times\left\{\gamma(x)\right\},\] where $\gamma=\ev_0-g$ ($\ev_t\colon C(I,\R{N})\to\R{N}$ stands for the evaluation at point $t\in I$). It is clear that the nonlocal Cauchy problem \eqref{nonlocal} is equivalent to the operator inclusion $Lx\in N_F(x)$.
\par For a given potential $V\in C^1(\R{N},\R{})$ we define the induced vector field $W_V\colon\R{N}\to\R{N}$ by the formula
\[W_V(x)=\begin{cases}
\nabla V(x),&\mbox{if }|\nabla V(x)|\<1,\\
\frac{\nabla V(x)}{|\nabla V(x)|},&\mbox{if }|\nabla V(x)|>1.
\end{cases}\]
Then $W_V$ is continuous and bounded. Let $N_{W_V}\colon X\to Z$ be such that  \[N_{W_V}(x)=(W_V(x(\cdot)),\gamma(x)).\] In this particular case we have $K_{P,Q}N_{W_V}(x)(t)=\int_0^tW_V(x(s))\,ds$. Therefore $N_{W_V}$ is $L$-compact (see \cite{maw} for more details).\par Throughout the rest of this paper $\Phi\colon\im Q\to\ker L$ will denote a fixed linear homeomorphism, given by $\Phi((0,x_0))=i(x_0)$. 
\par Regarding the guiding potential evoked in the title, we are obliged to introduce some auxiliary concepts.
We say that potential $V\colon\R{N}\to\R{}$ is {\it monotone}, if \[\left[|x|\<|y|\Rightarrow V(x)\<V(y)\right].\] Observe that, if potential $V$ is monotone, then $V$ satisfies $[|x|=|y|\Rightarrow V(x)=V(y)]$. In particular, $V$ is also even. The function $V\colon\R{N}\to\R{}$ is {\it coercive}, if $\lim\limits_{|x|\to+\infty}V(x)=+\infty$.
\begin{example}\mbox{ }
\begin{itemize}
\item[(i)] The classical Krasnosel'skii-Perov potential $V(x)=\frac{1}{2}|x|^2$ is monotone coercive and continuously differentiable.
\item[(ii)] Suppose $f\in C^1(\R{},\R{})$ is nondecreasing and there is $R>0$ such that $f(x)\geqslant x$ for $x\geqslant R$. Then the mapping $V(x)=f\left(\frac{1}{2}|x|^2\right)$ is a monotone coercive potential of $\,C^1$-class.
\end{itemize}
\end{example}
\par The method of guiding potential exploits also another generic notion:
\begin{definition}
A continuously differentiable function $V\colon\R{N}\to\R{}$ is called a nonsingular potential, if
\begin{equation}\label{nonsingular}
\exists\, R>0\;\forall\, |x|\geqslant R,\;\;\;\nabla V(x)\neq 0.
\end{equation}
\end{definition} 
\par The symbol $\langle A,B\rangle^{-}$ (or $\langle A,B\rangle^{+}$) we use below stands for the lower inner product (upper inner product) of nonempty compact subsets of $\R{n}$, i.e. 
\[\langle A,B\rangle^{-}=\inf\,\{\langle a,b\rangle\colon a\in A,b\in B\},\;\;\;\langle A,B\rangle^{+}=\sup\,\{\langle a,b\rangle\colon a\in A,b\in B\}.\] 
\par Let us specify the concept of a guiding potential for the multimap $F\colon I\times\R{N}\map\R{N}$ (compare \cite{blasi,gorn}):
\begin{definition}
Suppose that $V\in C^1(\R{N},\R{})$ is nonsingular. We will call the mapping $V$ 
\begin{itemize}
\item[] a weakly positively guiding potential for multimap $F$, if
\begin{equation}\label{weak} 
\exists\, R>0\;\forall\, |x|\geqslant R\;\forall\, t\in I\;\exists\, y\in F(t,x),\;\;\;\langle\nabla V(x),y\rangle\geqslant 0,
\end{equation}
\item[] a weakly negatively guiding potential for multimap $F$, if
\begin{equation}\label{negativ} 
\exists\, R>0\;\forall\, |x|\geqslant R\;\forall\, t\in I\;\exists\, y\in F(t,x),\;\;\;\langle\nabla V(x),y\rangle\leqslant 0,
\end{equation}
\item[] a strictly negatively guiding potential for multimap $F$, if
\begin{equation}\label{strictneg} 
\exists\, R>0\;\forall\, |x|\geqslant R\;\forall\, t\in I\;\;\;\langle\nabla V(x),F(t,x)\rangle^+<0.
\end{equation}
\end{itemize}
\end{definition}
In what follows we shall permanently refer to a certain initial set of assumptions regarding the multimap $F$, which concretizes the following definition:
\begin{definition}
We will say that a nonempty convex compact valued map $F\colon I\times\R{N}\map\R{N}$ is a Carath\'{e}odory map, if the following conditions are satisfied:
\begin{itemize}
\item[$(F_1)$] the multimap $t\mapsto F(t,x)$ is measurable for every fixed $x\in\R{N}$, 
\item[$(F_2)$] the multimap $x\mapsto F(t,x)$ is upper semicontinuous for $t\in I$ a.e.
\item[$(F_3)$] $\sup\{|y|\colon y\in F(t,x)\}\<\mu(t)(1+|x|)$ for every $(t,x)\in I\times\R{N}$, where $\mu\in L^1(I,\R{})$.
\end{itemize}
\end{definition}

\section{Main Results}
The first assertion illustrates why referring to guiding potential method is superfluous, if the function $g\colon C(I,\R{N})\to\R{N}$ satisfies sufficiently strong assumption regarding sublinear growth.
\begin{theorem}\label{th9}
Let $F\colon I\times\R{N}\map\R{N}$ be a Carath\'eodory map. Let $g\colon C(I,\R{N})\to\R{N}$ be a continuous mapping with contractive sublinear growth, i.e. 
\begin{equation}\label{sublinear}
\exists\,c\in(0,1)\;\exists\,d>0\;\forall\,x\in C(I,\R{N})\;\;\;|g(x)|\<c||x||+d.
\end{equation}
Then the nonlocal initial value problem \eqref{nonlocal} possesses at least one solution.
\end{theorem}
\begin{proof}
Define so-called solution set map $S_F\colon\R{N}\map C(I,\R{N})$, associated with the Cauchy problem 
\begin{equation}\label{cauchy}
\begin{cases}
\dot{x}(t)\in F(t,x(t)),&\mbox{a.e. }t\in I,\\
x(0)=x_0,
\end{cases}
\end{equation}
by the formula \[S_F(x_0):=\left\{x\in C(I,\R{N})\colon x\mbox{ is a solution of }\eqref{cauchy}\right\}.\] As it is well known the multimap $S_F$ is admissible (in terms of \cite[Def. 40.1]{gorn2}). In fact, it is an usc multimap with compact $R_\delta$ values. Now we are able to introduce the Poincar\'e-like operator $P\colon\R{N}\map\R{N}$ related to problem \eqref{nonlocal}, given by $P=g\circ S_F$. \par It is clear that $S_F(x_0)\subset\overline{B}(|x_0|+||\mu||_1)$ for every $x_0\in\R{N}$. Thus $S_F(\overline{B}(r))\subset\overline{B}(r+||\mu||_1)$. On the other hand we have $g(\overline{B}(M))\subset\overline{B}(cM+d)$, due to \eqref{sublinear}. Therefore $P(\overline{B}(r))\subset\overline{B}(c(r+||\mu||_1)+d)$. By the fact that $c<1$, we can take any $r$ such that \[r\geqslant\frac{c||\mu||_1+d}{1-c}>0\] to be sure that $P(\overline{B}(r))\subset\overline{B}(r)$. Of course, the multivalued operator $P\colon\overline{B}(r)\map\overline{B}(r)$ is compact admissible. As such, this operator possesses a fixed point $x_0\in\overline{B}(r)$ in view of the generalized Schauder fixed point theorem (\cite[Th.41.13]{gorn}). This point corresponds to the solution $x$ of the nonlocal Cauchy problem \eqref{nonlocal} in such a way that $x(0)=x_0=g(x)$.
\end{proof}
\begin{corollary}
Suppose all the assumptions of Theorem \ref{th9}. are satisfied. Then the set $S_{\!F}(g)$ of solutions to the problem \eqref{nonlocal} is nonempty and compact as a subset of $\,C(I,\R{N})$.
\end{corollary}
\begin{proof}
Retaining the notation of the proof of Theorem \ref{th9}., we claim that the fixed point set $\fix(P)$ is compact. Indeed, take $x_0\in\fix(P)$ and observe that $|x_0|=|g(x)|\<c||x||+d$ for some $x\in S_{\!F}(x_0)$. Thus, $|x_0|\<c(|x_0|+||\mu||_1)+d$ and as a result \[|x_0|\<\frac{c||\mu||_1+d}{1-c}.\] $\fix(P)$ is obviously closed, hence it must be compact.\par Define a multivalued map $\Psi\colon\fix(P)\map C(I,\R{N})$, by the formula \[\Psi(x_0)=\{x\in S_F(x_0)\colon g(x)=x_0\}.\] It easy to see that $\Psi$ is a compact valued usc multimap and that $S_{\!F}(g)$ coincides with $\Psi(\fix(P))$. Therefore, the solution set $S_{\!F}(g)$ is compact.
\end{proof}
\begin{remark}
The following exemplary classes of mappings have sublinear growth: compact maps, linear continuous maps, Lipschitzian maps, uniformly continuous maps. A sufficient condition for a map to have a contractive sublinear growth is, for instance: compactness, to be linear continuous with norm less then one, contractivity.
\end{remark}
The next theorem addresses the issue of the existence of solutions to the problem \eqref{nonlocal}, when mapping $g=-\ev_T$.
\begin{theorem}\label{th1}
If $F\colon I\times\R{N}\map\R{N}$ is a Carath\'eodory map and there exists an even coercive weakly positively guiding potential $V\colon\R{N}\to\R{}$ for the map $F$, then the antiperiodic problem 
\begin{equation}
\begin{cases}\label{anti}
\dot{x}(t)\in F(t,x(t)),&\mbox{a.e. }t\in I,\\
x(0)=-x(T)
\end{cases}
\end{equation}
possesses at least one solution.
\end{theorem}
\begin{proof}
Choose $R>0$ so that the conditions \eqref{nonsingular} and \eqref{weak} are satisfied. Fix \[r>\max\{V(x)\colon |x|\<R\}\] and put $G:=V^{-1}((-\infty,r)))$. Then $\Omega:=C(I,G)$ is open in $X$. Since $V$ is coercive, $\Omega$ is also bounded.\par Suppose that $x\in\overline{\Omega}=C(I,\overline{G})$ is a solution to the following antiperiodic problem
\[\begin{cases}
\dot{x}(t)=\lambda W_V(x(t)),&t\in I,\\
x(0)=-x(T)
\end{cases}\]
for some $\lambda\in(0,1)$. If there is $t_0\in(0,T)$ such that $x(t_0)\in\partial G$, then $t_0$ is a critical point of $V\circ x$ in view of Fermat lemma, i.e. $(V\circ x)'(t_0)=0$. On the other hand however $(V\circ x)'(t_0)=\langle\nabla V(x(t_0)),\lambda W_V(x(t_0))\rangle>0$, since $|x(t_0)|>R$. Thus $x(t_0)\in\partial G$ is contradicted. Assuming that $x(0)\in\partial G$ we see that \[\frac{1}{h}(V(x(h))-V(x(0)))\leqslant 0\] for every $h>0$. It means that the right-hand derivative $(V\circ x)'(0)\leqslant 0$. However, at the same time $\langle\nabla V(x(0)),\lambda W_V(x(0))\rangle>0$. Once again a contradiction. This implies that $x(0)\not\in\partial G$. Because $V$ is even the boundary $\partial G$ of the set $G$ is symmetric with respect to the origin. Thus it follows also that $x(T)\not\in\partial G$. Summing up, we see that $x(I)\cap\partial G=\emptyset$, i.e. $x\not\in\partial\Omega$. In fact, we have shown that $Lx\neq\lambda N_{W_V}(x)$ for every $x\in\dom L\cap\partial\Omega$ and $\lambda\in(0,1)$. \par It is easy to see that condition $0\not\in QN_{W_V}(\ker L\cap\partial\Omega)$ is equivalent to $x_0\neq-x_0$ for every $x_0\in\partial G$. The latter is obviously true, since $\partial G\subset\R{N}\setminus\overline{B}(R)\subset\R{N}\setminus\{0\}$.\par Let us notice that $i\colon i^{-1}(\Omega)\to\ker L\cap\Omega$ is a diffeomorphism of $C^1$-class. Using standard property of the Brouwer degree we get 
\begin{align*}
\deg(\Phi QN_{W_V},\ker L\cap\Omega,0)&=\deg(i^{-1}\Phi QN_{W_V}i,i^{-1}(\Omega),0)=\deg(\gamma\circ i,G,0)\\&=\deg((\ev_0+\ev_T)i,G,0).
\end{align*}
Observe that $(\ev_0+\ev_T)i(x)=2x$ and define $U\in C^1(\R{N},\R{})$, by $U(x)=|x|^2$. Since $U$ is a coercive potential we know, by \cite[Th.12.9.]{zabreiko}, that $\ind(U)=1$. Therefore \[\deg(\Phi QN_{W_V},\ker L\cap\Omega,0)=\deg(\nabla U,G,0)=\ind(U)\neq 0.\]
We are in position to apply the Continuation Theorem (Theorem \ref{conth}.) to deduce that 
\begin{equation}\label{wardeg}
\begin{split}
\deg(Id-P-(\Phi Q+K_{P,Q})N_{W_V},\Omega,0)&=\deg(Id-P-\Phi QN_{W_V},\Omega,0)\\&=\deg(-\Phi QN_{W_V},\ker L\cap\Omega,0)\\&\neq 0.
\end{split}
\end{equation}
Of course, the integer $\deg(Id-P-(\Phi Q+K_{P,Q})N_{W_V},\Omega,0)$ is nothing more than the coincidence degree $\deg((L,N_{W_V}),\Omega)$.
\par Following proof of Theorem 4.4. in \cite{gorn} we bring in a useful auxiliary multivalued map $F_V\colon I\times\R{N}\map\R{N}$ defined by: \[F_V(t,x)=F(t,x)\cap\{y\in\R{N}\colon \langle\alpha(x)\nabla V(x),y\rangle\geqslant 0\},\] where $\alpha$ is the Urysohn function
\[\alpha(x)=
\begin{cases}
0,&|x|\<R,\\
1,&|x|>R.
\end{cases}\]
So defined map $F_V$ has the property that $F_V(t,x)\subset F(t,x)$, while the function $V$ is a guiding potential for $F_V$ in the strict sense, i.e. 
\begin{equation}\label{war11}
\langle\nabla V(x),F_V(t,x)\rangle^-\geqslant 0
\end{equation}
for every $(t,x)\in I\times\R{N}$, with $|x|\geqslant R$. It is routine to check that $F_V$ is a Carath\'eodory map.\par Now we define another map $G\colon I\times\R{N}\times[0,1]\map\R{N}$ by \[G(t,x,\lambda):=\lambda W_V(x)+(1-\lambda)F_V(t,x).\] It is easy to see that $G$ is also a Carath\'eodory map. In particular, the map $G(t,\cdot,\cdot)$ is upper semicontinuous for a.a $t\in I$ and $\sup\{|y|\colon y\in G(t,x,\lambda)\}\<\max\{\mu(t),1\}(1+|x|)$. \par Define $N\colon\overline{\Omega}\times[0,1]\map Z$ by the formula \[N(x,\lambda):=\{f\in L^1(I,\R{N})\colon f(t)\in G(t,x(t),\lambda)\mbox{ for a.a. }t\in I\}\times\{\gamma(x)\},\] where $\gamma=\ev_0+\ev_T$. Now we may introduce a homotopy $H\colon\overline{\Omega}\times[0,1]\map X$ in the following way: 
\begin{equation}\label{homotopy}
H(x,\lambda):=Px+(\Phi Q+K_{P,Q})N(x,\lambda).
\end{equation}
Observe that $QN(x,\lambda)=(0,\gamma(x))$ and $K_{P,Q}N(x,\lambda)(t)=\int_0^tG(s,x(s),\lambda)\,ds$. It is clear that $QN$ is completely continuous. Standard and plain arguments justify that $K_{P,Q}N$ is an usc multimap with compact convex values and the range $K_{P,Q}N(\overline{\Omega}\times[0,1])$ is relatively compact. Bearing in mind that $P$ is linear continuous and has a finite dimensional range it is clear that $P(\overline{\Omega})$ is also relatively compact. All these observations allow to conclude that the homotopy $H$ is a compact usc multimap with convex compact values.\par We claim that $x\not\in H(x,\lambda)$ on $\partial\Omega\times[0,1]$. Let $x\in\overline{\Omega}=C(I,\overline{G})$ be a solution to the problem  
\[\begin{cases}
\dot{x}(t)\in G(t,x(t),\lambda),&\mbox{a.e. }t\in I,\\
x(0)=-x(T)
\end{cases}\]
for some $\lambda\in(0,1]$. \par Suppose there is $t_0\in[0,T)$ such that $x(t_0)\in\partial G$. This means that $|x(t_0)|>R$ and as a result there is $\delta\in(0,T-t_0]$ such that $|x(t)|>R$ for $t\in[t_0,t_0+\delta]$. It implies that
\begin{equation}\label{warprod}
\begin{split}
\langle\nabla V(x(t)),\dot{x}(t)\rangle&\geqslant\langle\nabla V(x(t)),\lambda W_V(x(t))+(1-\lambda)F_V(t,x(t))\rangle^-\\&\geqslant\lambda\langle\nabla V(x(t)),W_V(x(t))\rangle+(1-\lambda)\langle\nabla V(x(t)),F_V(t,x(t))\rangle^-\\&\geqslant\lambda\langle\nabla V(x(t)),W_V(x(t))\rangle\hspace{0.5cm}\mbox{ (by \eqref{war11})}\\&>0\hspace{0.5cm}\mbox{ (by \eqref{nonsingular})}
\end{split}
\end{equation}
for almost all $t\in[t_0,t_0+\delta]$. Therefore we get: \[V(x(t_0+\delta))-V(x(t_0))=\int_{t_0}^{t_0+\delta}\langle\nabla V(x(t)),\dot{x}(t)\rangle\,dt>0\] and hence $V(x(t_0))<V(x(t_0+\delta))\<r$. We arrive at contradiction with $x(t_0)\in\partial G$. Once more we use the fact that $V$ is even to infer that $x(T)\not\in\partial G$. Summing up, $x\not\in\partial\Omega$. In fact, the above reasoning proves that $Lx\not\in N(x,\lambda)$ for all $x\in\dom L\cap\partial\Omega$ and $\lambda\in(0,1]$. Furthermore, suppose that $Lx\not\in N(x,0)$ for every $x\in\partial\Omega$ (otherwise \eqref{anti} has a solution and there is nothing to prove). Taking into account that $Lx\in N(x,\lambda)\Leftrightarrow x\in Px+(\Phi Q+K_{P,Q})N(x,\lambda)$ we conclude that $x\not\in H(x,\lambda)$ on $\partial\Omega\times[0,1]$ is verified.\par Relying on the homotopy invariance of the topological degree for the class ${\mathcal M}$ (or if the Reader considers it more appriopriate, of coincidence degree theory) we obtain the equality 
\begin{align*}
\deg(Id-P-(\Phi Q+K_{P,Q})N(\cdot,0),\Omega,0)&=\deg(Id-H(\cdot,0),\Omega,0)=\deg(Id-H(\cdot,1),\Omega,0)\\&=\deg(Id-P-(\Phi Q+K_{P,Q})N_{W_V},\Omega,0).
\end{align*}
In view of \eqref{wardeg} we see that $\deg(Id-P-(\Phi Q+K_{P,Q})N(\cdot,0),\Omega,0)\neq 0$. This means that there is a coincidence point $x\in\Omega\cap\dom L$ of the inclusion $Lx\in N(x,0)$. This point constitutes a solution to the problem \eqref{anti}.
\end{proof}
\begin{corollary}
Assume that $F\colon I\times\R{N}\map\R{N}$ is a Carath\'eodory map and there exists an even coercive weakly negatively guiding potential $V\colon\R{N}\to\R{}$ for the map $F$. Then the antiperiodic problem \eqref{anti} has at least one solution.
\end{corollary}
\begin{proof}
Define a map $W\colon\R{N}\to\R{}$, by $W(x)=-V(x)$. Then $W$ is an even coercive weakly positively guiding potential for the map $F$. Thus, the thesis of Theorem \ref{th1}. applies.
\end{proof}
\begin{remark}
Let us notice that the coercivity condition in Theorem \ref{th1}. could not be dropped. In fact,
\[V(x)\underset{|x|\to+\infty}{\longrightarrow}+\infty\Longleftrightarrow\forall\,r>0\;\;\;V^{-1}((-\infty,r))\;\mbox{ is bounded}.\]
\end{remark}
The following theorem refers to the case when the function $g$ takes the form of a linear combination of evaluations at fixed points of division of segment $I$. Note that the anti-periodicity condition is a special case of the multi-point discrete mean condition.
\begin{theorem}\label{th2}
Let $0<t_1<t_2<\ldots<t_n\<T$ be arbitrary, but fixed, $\sum_{i=1}^n|\alpha_i|\<1$ and $\sum_{i=1}^n\alpha_i\neq 1$. If $F\colon I\times\R{N}\map\R{N}$ is a Carath\'eodory map and there exists a monotone coercive weakly negatively guiding potential $V\colon\R{N}\to\R{}$ for the map $F$, then the following nonlocal initial value problem with multi-point discrete mean condition
\begin{equation}
\begin{cases}\label{multi}
\dot{x}(t)\in F(t,x(t)),&\mbox{a.e. }t\in I,\\
x(0)=\sum\limits_{i=1}^n\alpha_ix(t_i)
\end{cases}
\end{equation}
possesses at least one solution.
\end{theorem}
\begin{proof}
We keep the notations introduced in the proof of Theorem \ref{th1}. Let $x\in\overline{\Omega}$ be such that
\[\begin{cases}
\dot{x}(t)=-\lambda W_V(x(t)),&t\in I,\\
x(0)=\sum\limits_{i=1}^n\alpha_ix(t_i),
\end{cases}\]
for some $\lambda\in(0,1)$. Taking into account that $\langle\nabla V(x(t)),-\lambda W_V(x(t))\rangle<0$ for all $t\in I$, with $|x(t)|>R$, we infer in a strictly analogous manner to the proof of Theorem \ref{th1} that $x(t)\not\in\partial G$ for $t\in(0,T)$. If $x(T)\in\partial G$, then the left-hand derivative $(V\circ x)'(T)\geqslant 0$, which is in contradiction with $\langle\nabla V(x(T)),-\lambda W_V(x(T))\rangle<0$. Notice that
\[|x(0)|=\left|\sum\limits_{i=1}^n\alpha_ix(t_i)\right|\<\sum\limits_{i=1}^n|\alpha_i||x(t_i)|\<\max_{1\<i\<n}|x(t_i)|=|x(t_{i_0})|,\]
where $t_{i_0}\in(0,T]$. Keeping in mind that $V$ is nondecreasing we get $V(x(0))\<V(x(t_{i_0}))<r$. Thus $x(0)\not\in\partial G$ and $x\not\in\partial\Omega$. These considerations show that $Lx\neq\lambda N_{(-W_V)}(x)$ for every $x\in\dom L\cap\partial\Omega$ and $\lambda\in(0,1)$. \par Clearly, $0\not\in QN_{(-W_V)}(\ker L\cap\partial\Omega)$ is equivalent to $\gamma(x)\neq 0$ for $x\in\ker L\cap\partial\Omega$. In the current case $\gamma=\ev_0-\sum_{i=1}^n\alpha_i\ev_{t_i}$. Therefore, we demand that $\ev_0(i(x_0))\neq\sum_{i=1}^n\alpha_i\ev_{t_i}(i(x_0))$, i.e. $x_0\neq\sum_{i=1}^n\alpha_ix_0$ for every $x_0\in\partial G$. The latter is true, since $\sum_{i=1}^n\alpha_i\neq 1$.\par Put $W(x)=(1-\sum_{i=1}^n\alpha_i)\frac{1}{2}|x|^2$. Applying again Continuation Theorem we get 
\begin{equation}\label{warminus}
\begin{split}
\deg(Id-P&-(\Phi Q+K_{P,Q})N_{(-W_V)},\Omega,0)=\deg(Id-P-\Phi QN_{(-W_V)},\Omega,0)\\&=\deg(-\Phi QN_{(-W_V)},\ker L\cap\Omega,0)=\deg(-\gamma\circ i,G,0)\\&=\deg\left(-\left(\ev_0-\sum_{i=1}^n\alpha_i\ev_{t_i}\right)i,G,0\right)=\deg(-\nabla W,G,0)\\&=(-1)^{N+1}\deg(\nabla W,B(R),0)=(-1)^{N+1}\ind(W)\neq 0.
\end{split}
\end{equation}
\par Let us modify definition of the multimap $F_V\colon I\times\R{N}\map\R{N}$ in the following way: 
\begin{equation}\label{fvbis}
F_V(t,x)=F(t,x)\cap\{y\in\R{N}\colon \langle\alpha(x)\nabla V(x),y\rangle\leqslant 0\}.
\end{equation}
Then \eqref{war11} takes the form
\begin{equation}\label{FV}
\langle\nabla V(x),F_V(t,x)\rangle^+\leqslant 0
\end{equation}
for every $(t,x)\in I\times\R{N}$, with $|x|\geqslant R$. We change respectively definition of multimap $G\colon I\times\R{N}\times[0,1]\map\R{N}$, namely \[G(t,x,\lambda):=\lambda (-W_V)(x)+(1-\lambda)F_V(t,x).\] Our aim is to show that $x\not\in H(x,\lambda)$ on $\partial\Omega\times[0,1]$, where $H$ is a homotopy defined by \eqref{homotopy}.\par Let $x\in\overline{\Omega}$ be a solution to the nonlocal Cauchy problem  
\[\begin{cases}
\dot{x}(t)\in G(t,x(t),\lambda),&\mbox{a.e. }t\in I,\\
x(0)=\sum_{i=1}^n\alpha_ix(t_i)
\end{cases}\]
for some $\lambda\in(0,1]$. If there is $t_0\in(0,T]$ such that $x(t_0)\in\partial G$, then there is $\delta\in(0,t_0]$ such that $|x(t)|>R$ for $t\in[t_0-\delta,t_0]$. Now \eqref{warprod} assumes the form
\[\begin{split}
\langle\nabla V(x(t)),\dot{x}(t)\rangle&\leqslant\langle\nabla V(x(t)),\lambda(-W_V(x(t)))+(1-\lambda)F_V(t,x(t))\rangle^+\\&\leqslant\lambda\langle\nabla V(x(t)),-W_V(x(t))\rangle+(1-\lambda)\langle\nabla V(x(t)),F_V(t,x(t))\rangle^+\\&\leqslant\lambda\langle\nabla V(x(t)),-W_V(x(t))\rangle\hspace{0.5cm}\mbox{ (by \eqref{FV})}\\&<0
\end{split}\]
for almost all $t\in[t_0-\delta,t_0]$. Consequently \[V(x(t_0))-V(x(t_0-\delta))=\int_{t_0-\delta}^{t_0}\langle\nabla V(x(t)),\dot{x}(t)\rangle\,dt<0.\] Thus $V(x(t_0))<V(x(t_0-\delta))\<r$ and $x(t_0)\in\partial G$ is contradicted. Recall that the boundary condition implies $|x(0)|\<|x(t_{i_0})|$ for some $t_{i_0}\in(0,T]$. Thus $V(x(0))<r$, since $V$ is monotone. Consequently $x\not\in\partial\Omega$. Analogous reasoning like in the proof of Theorem \ref{th1}. justifies $x\not\in H(x,\lambda)$ on $\partial\Omega\times[0,1]$.\par Applying the homotopy invariance and property \eqref{warminus} we arrive at the conclusion
\begin{align*}
\deg(Id-P-(\Phi Q+K_{P,Q})N(\cdot,0),\Omega,0)&=\deg(Id-H(\cdot,0),\Omega,0)=\deg(Id-H(\cdot,1),\Omega,0)\\&=\deg(Id-P-(\Phi Q+K_{P,Q})N_{(-W_V)},\Omega,0)\neq 0.
\end{align*}
This indicates that there is a coincidence point $x\in\Omega\cap\dom L$ of the operator inclusion $Lx\in N(x,0)$, which means that \eqref{multi} has at least one solution.
\end{proof}
Substitution of the notion of a weakly negatively guiding potential in the hypotheses of Theorem \ref{th2}. leads to the following conclusion:
\begin{corollary}
Let $0\<t_1<t_2<\ldots<t_n<T$ be arbitrary, but fixed, $\sum_{i=1}^n|\alpha_i|\<1$ and $\sum_{i=1}^n\alpha_i\neq 1$. If $F\colon I\times\R{N}\map\R{N}$ is a Carath\'eodory map and there exists a monotone coercive weakly positively guiding potential $V\colon\R{N}\to\R{}$ for the map $F$, then the following nonlocal boundary value problem with multi-point discrete mean condition
\[\begin{cases}
\dot{x}(t)\in F(t,x(t)),&\mbox{a.e. }t\in I,\\
x(T)=\sum\limits_{i=1}^n\alpha_ix(t_i)
\end{cases}\]
possesses at least one solution.
\end{corollary}
Subsequent result concerns the situation where the mapping g, determining the boundary condition, has the form of mean value of the composition of its argument and some subsidiary function $h$. Unfortunately, the argumentation contained in the proof of this theorem, based on the evaluation of the Brouwer degree of $Id-h$, excludes the case $g(x)=\frac{1}{T}\int_0^Tx(t)\,dt$. 
\begin{theorem}\label{th3}
Let $h\colon\R{N}\to\R{N}$ be a continuous mapping such that $|h(x)|\<|x|$ for every $x\in\R{N}$. Assume further that the fixed point set of $h$ is compact. If $F\colon I\times\R{N}\map\R{N}$ is a Carath\'eodory map and there exists a monotone coercive weakly negatively guiding potential $V\colon\R{N}\to\R{}$ for the map $F$, then the following nonlocal Cauchy problem with mean value condition
\begin{equation}
\begin{cases}\label{mean}
\dot{x}(t)\in F(t,x(t)),&\mbox{a.e. }t\in I,\\
x(0)=\frac{1}{T}\int_0^Th(x(t))\,dt&
\end{cases}
\end{equation}
possesses at least one solution.
\end{theorem}
\begin{proof}
The fixed point set $\fix(h)$ is compact iff there is $M>0$ such that $\fix(h)\subset B(M)$. Choose $R\geqslant M$ according to \eqref{nonsingular} and \eqref{negativ}. Following the scheme presented in the proof of Theorem \ref{th2}., take $x\in\overline{\Omega}$ such that
\[\begin{cases}
\dot{x}(t)=-\lambda W_V(x(t)),&t\in I,\\
x(0)=\frac{1}{T}\int_0^Th(x(t))\,dt,
\end{cases}\]
for some $\lambda\in(0,1)$. We know already that $x(t)\not\in\partial G$ for $t\in(0,T]$. Suppose $|x(0)|>|x(t)|$ for $t\in(0,T]$. Then
\[|x(0)|=\frac{1}{T}\int_0^T|x(0)|\,dt>\frac{1}{T}\int_0^T|x(t)|\,dt\geqslant\frac{1}{T}\int_0^T|h(x(t))|\,dt\geqslant\left|\frac{1}{T}\int_0^Th(x(t))\,dt\right|=|x(0)|.\] Therefore there exists $t_0\in(0,T]$ such that $|x(0)|\<|x(t_0)|$. Since $V$ is monotone we have $V(x(0))\<V(x(t_0))<r$, i.e. $x(0)\not\in\partial G$. This proves that $x\not\in\partial\Omega$ resulting in: $Lx\neq\lambda N_{(-W_V)}(x)$ for every $x\in\dom L\cap\partial\Omega$ and $\lambda\in(0,1)$. \par Condition $0\not\in QN_{(-W_V)}(\ker L\cap\partial\Omega)$ is equivalent to $x_0\neq\frac{1}{T}\int_0^Th(i(x_0)(t))\,dt$, i.e. $x_0\neq h(x_0)$ for $x_0\in\partial G$. The latter requirement is fulfilled in our case, because \[x_0\in\partial G\Rightarrow |x_0|>R\Rightarrow x_0\not\in\overline{B}(M)\Rightarrow x_0\not\in\fix(h)\Rightarrow x_0\neq h(x_0).\]
\par Take $x_0\in\partial G$. Then $|x_0-h(x_0))|>0$ and $|h(x_0)|\<|x_0|$. Thus $|x_0-h(x_0))-x_0|=|h(x_0)|<|x_0-h(x_0)|+|x_0|$. The last inequality is an equivalent formulation of the Poincar\'e-Bohl theorem (\cite[Th.2.1.]{zabreiko}), which means that
$\lambda(Id-h)(x_0)+(1-\lambda)Id(x_0)\neq 0$ for every $(x_0,\lambda)\in\partial G\times[0,1]$, i.e. vector fields $Id-h$ and $Id$ are joined by the linear homotopy, which has no zeros on the boundary $\partial G$. Apply Continuation Theorem to see that
\begin{equation}\label{warnonzero2}
\begin{split}
\deg(Id-P&-(\Phi Q+K_{P,Q})N_{(-W_V)},\Omega,0)=\deg(Id-P-\Phi QN_{(-W_V)},\Omega,0)\\&=\deg(-\Phi QN_{(-W_V)},\ker L\cap\Omega,0)=\deg(-(Id-h),G,0)\\&=(-1)^{N+1}\deg(Id,G,0)=(-1)^{N+1}\neq 0,
\end{split}
\end{equation}
as $0\in B(R)\subset G$.\par Following proof of Theorem \ref{th2}. we consider a solution $x\in\overline{\Omega}$ of the succeeding boundary value problem  
\[\begin{cases}
\dot{x}(t)\in G(t,x(t),\lambda),&\mbox{a.e. }t\in I,\\
x(0)=\frac{1}{T}\int_0^Th(x(t))\,dt,
\end{cases}\]
for some $\lambda\in(0,1]$. The reasoning goes without changes, so that we show that $x(t)\not\in\partial G$ for $t\in(0,T]$. The boundary condition as before implies that $x(0)\not\in\partial G$. Therefore, $x\not\in\partial\Omega$, which means that the homotopy $H$ has no fixed points on $\partial\Omega\times[0,1]$. The homotopy invariance together with \eqref{warnonzero2} entails
\begin{align*}
\deg(Id-P-(\Phi Q+K_{P,Q})N(\cdot,0),\Omega,0)&=\deg(Id-H(\cdot,0),\Omega,0)=\deg(Id-H(\cdot,1),\Omega,0)\\&=\deg(Id-P-(\Phi Q+K_{P,Q})N_{(-W_V)},\Omega,0)\neq 0.
\end{align*}
Summarizing, there exists a coincidence point $x\in\Omega\cap\dom L$ of the operator inclusion $Lx\in N(x,0)$. This point constitutes the solution of the nonlocal Cauchy problem \eqref{mean} with mean value condition.
\end{proof}
\begin{corollary}
Let $h\colon\R{N}\to\R{N}$ be a continuous mapping such that $|h(x)|\<|x|$ for every $x\in\R{N}$. Assume further that the fixed point set of $h$ is compact. If $F\colon I\times\R{N}\map\R{N}$ is a Carath\'eodory map and there exists a monotone coercive weakly positively guiding potential $V\colon\R{N}\to\R{}$ for the map $F$, then the following nonlocal boundary value problem with mean value condition
\[\begin{cases}
\dot{x}(t)\in F(t,x(t)),&\mbox{a.e. }t\in I,\\
x(T)=\frac{1}{T}\int_0^Th(x(t))\,dt&
\end{cases}\]
possesses at least one solution.
\end{corollary}
A cursory look at the proof of Theorem \ref{th2}. and \ref{th3}. leads to the following generalization.
\begin{theorem}\label{th4}
Assume that $F\colon I\times\R{N}\map\R{N}$ is a Carath\'eodory map and there exists a monotone coercive weakly negatively guiding potential $V\colon\R{N}\to\R{}$ for the map $F$. Let $R>0$ be chosen with accordance to \eqref{nonsingular} and \eqref{negativ} and $r>\max\{V(x)\colon |x|\<R\}$. Let $g\colon C(I,\R{N})\to\R{N}$ be a continuous function, which maps bounded sets into bounded sets and satisfies additionally the following conditions:
\begin{itemize}
\item[(i)] $\forall\, x\in\dom L\cap C(I,V^{-1}((-\infty,r]))\;\;\;[x(0)=g(x)]\Rightarrow\exists\,t\in(0,T]\;\;\;|g(x)|\<|x(t)|$,
\item[(ii)] $|g(i(x))|\<|x|\;$ for all $x\in V^{-1}(\{r\})$,
\item[(iii)] $\fix(g\circ i)\cap V^{-1}(\{r\})=\emptyset$.
\end{itemize}
Then the nonlocal Cauchy problem \eqref{nonlocal} possesses at least one solution.
\end{theorem}
\begin{proof}
Our reasoning will reproduce exactly the proceedings of the proof of Theorem \ref{th2}. Let $x\in\overline{\Omega}$ be such that
\[\begin{cases}
\dot{x}(t)=-\lambda W_V(x(t)),&t\in I,\\
x(0)=g(x),
\end{cases}\]
for some $\lambda\in(0,1)$. We infer in a strictly analogous manner to the proof of Theorem \ref{th2}. that $x(t)\not\in\partial G$ for $t\in(0,T]$. Notice that $\overline{G}=V^{-1}((-\infty,r])$, due to \eqref{nonsingular}. Thus, $x\in\dom L\cap C(I,V^{-1}((-\infty,r]))$. It follows from condition (i) that there exists $t\in(0,T]$ such that $|x(0)|\<|x(t)|$. Since $V$ is monotone we get $V(x(0))\<V(x(t))<r$. Therefore $x(0)\not\in\partial G$ and $x\not\in\partial\Omega$. Consequently, $Lx\neq\lambda N_{(-W_V)}(x)$ for every $x\in\dom L\cap\partial\Omega$ and $\lambda\in(0,1)$. \par Recall that $0\not\in QN_{(-W_V)}(\ker L\cap\partial\Omega)$ is equivalent to $\ev_0(x)\neq g(x)$ for $x\in\ker L\cap\partial\Omega$. In other words: $x_0\neq g(i(x_0))$ for every $x_0\in\partial G=V^{-1}(\{r\})$. Of course, condition (iii) is formulated so that the latter is true. 
\par Take $x_0\in\partial G$. Then $|x_0-g(i(x_0))|>0$, by condition (iii) and $|g(i(x_0))|\<|x_0|$, by condition (ii). Whence $|x_0-g(i(x_0))-x_0|<|x_0-g(i(x_0))|+|x_0|$. The last inequality is a simple reformulation of the Poincar\'e-Bohl condition (\cite[Th.2.1.]{zabreiko}), which means that
$\lambda(Id-g\circ i)(x_0)+(1-\lambda)Id(x_0)\neq 0$ for every $(x_0,\lambda)\in\partial G\times[0,1]$, i.e. vector fields $Id-g\circ i$ and $Id$ are joined by the homotopy, nonsingular on the boundary $\partial G$. Now we can apply Continuation Theorem:
\begin{equation}\label{warnonzero}
\begin{split}
\deg(Id-P&-(\Phi Q+K_{P,Q})N_{(-W_V)},\Omega,0)=\deg(Id-P-\Phi QN_{(-W_V)},\Omega,0)\\&=\deg(-\Phi QN_{(-W_V)},\ker L\cap\Omega,0)=\deg(-(Id-g\circ i),G,0)\\&=(-1)^{N+1}\deg(Id,G,0)=(-1)^{N+1}\neq 0,
\end{split}
\end{equation}
as $0\in B(R)\subset G$.\par The rest of the proof is strictly analogous to the remaining part of the proof of Theorem \ref{th2}. In particular, if $x\in\overline{\Omega}$ is a solution to the problem
\[\begin{cases}
\dot{x}(t)\in G(t,x(t),\lambda),&\mbox{a.e. }t\in I,\\
x(0)=g(x)
\end{cases}\]
for some $\lambda\in(0,1]$, then the only modification would be the use of property (i) to demonstrate that $x(0)\not\in\partial G$. In this way we prove that $x\not\in\partial\Omega$, which means that $x\not\in H(x,\lambda)$ on $\partial\Omega\times[0,1]$. By reffering to the property \eqref{warnonzero} we conclude that
\begin{align*}
\deg(Id-P-(\Phi Q+K_{P,Q})N(\cdot,0),\Omega,0)&=\deg(Id-H(\cdot,0),\Omega,0)=\deg(Id-H(\cdot,1),\Omega,0)\\&=\deg(Id-P-(\Phi Q+K_{P,Q})N_{(-W_V)},\Omega,0)\neq 0.
\end{align*}
The coincidence point $x\in\Omega\cap\dom L$ of the operator inclusion $Lx\in N(x,0)$ determines the solution of the nonlocal Cauchy problem \eqref{nonlocal}.
\end{proof}
\begin{corollary}\label{wn2}
Assume that $F\colon I\times\R{N}\map\R{N}$ is a Carath\'eodory map and there exists a monotone coercive weakly positively guiding potential $V\colon\R{N}\to\R{}$ for the map $F$. Let $R>0$ be chosen with accordance to \eqref{nonsingular} and \eqref{weak} and $r>\max\{V(x)\colon |x|\<R\}$. 
Let $g\colon C(I,\R{N})\to\R{N}$ be a continuous function, which maps bounded sets into bounded sets and satisfies conditions {\em (ii)-(iii)} of Theorem \ref{th4}. along with:
\begin{itemize}
\item[(iv)] $\forall\, x\in\dom L\cap C(I,V^{-1}((-\infty,r]))\;\;\;[x(T)=g(x)]\Rightarrow\exists\,t\in(0,T]\;\;\;|g(x)|\<|x(t)|$.
\end{itemize}
Then the nonlocal boundary value problem
\[\begin{cases}
\dot{x}(t)\in F(t,x(t)),&\mbox{a.e. }t\in I,\\
x(T)=g(x)
\end{cases}\]
possesses at least one solution.
\end{corollary}
\begin{example}
The mappings, which were used to define nonlocal initial conditions in Theorems \ref{th1}.,\ref{th2}. and \ref{th3}. satisfy conditions {\em (i)-(iii)}. 
\end{example}
\begin{example}
The following requirement
\[\forall\,0\neq x\in\dom L\cap C(I,V^{-1}((-\infty,r]))\;\exists\,t\in(0,T]\;\;\;|g(x)|<|x(t)|.\]
implies conditions {\em (i)-(iii)}.
\end{example}
\begin{example}
If $g\colon C(I,\R{N})\to\R{N}$ is such that $|g(x)|<||x||$ for every $x\neq 0$, then in particular conditions {\em (i)-(iii)} are satisfied.
\end{example}
All previous theorems apply to the cases where the function $g$ satisfies the estimation $|g(x)|\<||x||$ for every $x\in C(I,\R{N})$. The next result is based on such a property of $g$, which prevents the function $g$ from fulfillment of this estimation.
\begin{theorem}\label{th6}
Let $F\colon I\times\R{N}\map\R{N}$ be a Carath\'eodory map. Assume that $g\colon C(I,\R{N})\to\R{N}$ is continuous, maps bounded sets into bounded sets and satisfies the following conditions:
\begin{itemize}
\item[(a)] $\liminf\limits_{||x||\to+\infty}\frac{|g(x)|}{||x||}>1$,
\item[(b)] the counter image $(g\circ i)^{-1}(\{0\})$ is compact and $\deg(g\circ i,B(R),0)\neq 0$, where $(g\circ i)^{-1}(\{0\})\subset B(R)$.
\end{itemize}
Then the solution set $S_{\!F}(g)$ of nonlocal initial value problem \eqref{nonlocal} is nonempty and compact as a subset of $C(I,\R{N})$.
\end{theorem}
\begin{proof}
Let $R>0$ be such that $(g\circ i)^{-1}(\{0\})\subset B(R)$. Property (a) of the mapping $g$ means that there is $M>0$ such that $|g(x)|>||x||$ for all $||x||\geqslant M$. Assume that $M\geqslant R$ and put $\Omega:=C(I,B(M))$. Define $N_F\colon\overline{\Omega}\map Z$, by \[N_F(x):=\left\{f\in L^1(I,\R{N})\colon f(t)\in F(t,x(t))\mbox{ for a.a. }t\in I\right\}\times\left\{\gamma(x)\right\},\] where $\gamma=\ev_0-g$. Taking into account that $QN_F(x)=(0,\gamma(x))$ and $K_{P,Q}N_F(x)(t)=\int_0^tF(s,x(s))\,ds$ we infer that $QN_F$ and $K_{P,Q}N_F$ are compact usc multimaps with compact convex values.\par Suppose $x\in\overline{\Omega}$ is a solution to the following nonlocal Cauchy problem
\[\begin{cases}
\dot{x}(t)\in\lambda F(t,(x(t)),&\mbox{a.e. }t\in I,\\
x(0)=g(x)
\end{cases}\]
for some $\lambda\in(0,1)$. From (a) it follows that $||x||<M$. Otherwise $x(0)\neq g(x)$. Thus $x\not\in\partial\Omega$ and inclusion $Lx\in\lambda N_F(x)$ has no solution on $\partial\Omega$ for every $\lambda\in(0,1)$.\par Observe that $0\not\in QN_F(\ker L\cap\partial\Omega)$ is equivalent to $x_0\neq g(i(x_0))$ for every $x_0\in\partial B(M)$. However \[x_0\in\partial B(M)\Rightarrow |x_0|=||i(x_0)||\geqslant M\Rightarrow |g(i(x_0))|>|x_0|\Rightarrow x_0\neq g(i(x_0)).\]
\par Take $x_0\in\partial B(M)$. Then $|x_0|<|g(i(x_0))|$. In other words $|x_0-g(i(x_0))-(-g(i(x_0)))|<|-g(i(x_0))|$. It follows from a corollary to the theorem of Poincar\'e-Bohl (\cite[Th.2.3.]{zabreiko}) that vector fields $Id-g\circ i$ and $-g\circ i$ are mutually homotopic. Using standard properties of the Brouwer degree we see that 
\begin{align*}
\deg(id-g\circ i,B(M),0)&=\deg(-g\circ i,B(M),0)=(-1)^{N+1}\deg(g\circ i,B(M),0)\\&=(-1)^{N+1}\deg(g\circ i,B(R),0)\neq 0,
\end{align*}
by (b). Applying the same reasoning as in the proof of Theorem \ref{th2}. we get
\[\deg(\Phi QN_F,\ker L\cap\Omega,0)=\deg(\gamma\circ i,B(M),0)=\deg(Id-g\circ i,B(M),0)\neq 0.\]
In view of Theorem \ref{conth}. there is a solution of the inclusion $Lx\in N_F(x)$ in $\Omega$. This is of course also a solution of the problem \eqref{nonlocal}.\par It is easy to see that compactness of $S_{\!F}(g)$ is equivalent to the existence of a priori bounds on the solutions to \eqref{nonlocal}. Indeed, suppose $S_{\!F}(g)$ is bounded. Then this set must also be equicontinuous in view of $(F_3)$. Therefore it is relatively compact. The closedness of $S_{\!F}(g)$ is a straightforward consequence of Compactness Theorem (\cite[Th.0.3.4.]{aubin}), Convergence Theorem (\cite[Th.1.4.1.]{aubin}) and continuity of $g$. \par Suppose that for every $r>0$ there exists an element $x\in S_{\!F}(g)$ such that $||x||>r$. Choose $r:=M$. Then $|g(x)|>||x||\geqslant|x(0)|$ for some $x\in S_{\!F}(g)$, which yields a contradiction with $x(0)=g(x)$. Therefore the solution set $S_{\!F}(g)$ must be bounded, completing the proof.
\end{proof}
\begin{lemma}\label{lem2}
Suppose $F\colon I\times\R{N}\map\R{N}$ has sublinear growth, i.e. $(F_3)$ is satisfied. Let $x\in AC(I,\R{N})$ be a solution to the following Cauchy problem
\[\begin{cases}
\dot{x}(t)\in F(t,x(t)),&\mbox{a.e. on }I,\\
x(0)=x_0.
\end{cases}\]
Then $|x(t)|>r>0$ for every $t\in I$, provided $|x_0|>e^{||\mu||_1}(r+1)-1$.
\end{lemma}
\begin{proof}
Taking into account that $x\in AC(I,\R{N})$ and $|\dot{x}(t)|\<\mu(t)(1+|x(t)|)$ a.e. on $I$, in wiev of $(F_2)$ we have \[|x(t)-x(s)|\<\int_s^t\mu(\tau)(1+|x(\tau)|)\,d\tau\] for $\,0\<s\<t\<T$. Thus \[|x(t)|\geqslant|x(s)|-\int_s^t\mu(\tau)(1+|x(\tau)|)\,d\tau\] for $\,0\<s\<t\<T$. Define $f\colon[0,t]\to\R{}_+$ by \[f(s)=|x(t)|+\int_s^t\mu(\tau)(1+|x(\tau)|)\,d\tau+1.\] Then
\begin{equation}\label{lem1}
f(s)\geqslant|x(s)|+1>0\mbox{ and }f(t)=|x(t)|+1.
\end{equation}
Obviously \[\ln f(t)-\ln f(s)=\int_s^t\frac{f'(\tau)}{f(\tau)}\,d\tau.\] Same time \[f'(s)=-\mu(s)(1+|x(s)|)\] a.e. on $[0,t]$. Therefore 
\begin{align*}
\ln\frac{f(t)}{f(s)}&=\ln f(t)-\ln f(s)=-\int_s^t\frac{\mu(\tau)(1+|x(\tau)|)}{f(\tau)}\,d\tau\overset{\eqref{lem1}}{\geqslant}-\int_s^t\frac{\mu(\tau)(1+|x(\tau)|)}{|x(\tau)|+1}\,d\tau\\&=-\int_s^t\mu(\tau)\,d\tau.
\end{align*}
From here \[f(t)\geqslant f(s)\exp\left(-\int_s^t\mu(\tau)\,d\tau\right)\] which implies \[|x(t)|+1\geqslant(|x(s)|+1)\exp\left(-\int_s^t\mu(\tau)\,d\tau\right)\] for $\,0\<s\<t\<T$, by \eqref{lem1}. Substituting $s=0$ we have \[|x(t)|\geqslant(|x_0|+1)\exp\left(-\int_0^t\mu(\tau)\,d\tau\right)-1\geqslant(|x_0|+1)e^{-||\mu||_1}-1.\] Now the assertion of the lemma is already visible.
\end{proof}
The purpose of the last theorem is to provide the conditions that ensure the compactness of the set of solutions to nonlocal Cauchy problems that were the subject of interest in Theorems \ref{th1}., \ref{th2}. and \ref{th3}.
\begin{theorem}\label{compact}
Assume that $F\colon I\times\R{N}\map\R{N}$ is a Carath\'eodory map. Suppose there exists a monotone coercive strictly negatively guiding potential $V\colon\R{N}\to\R{}$ for the map $F$. Let $R>0$ be chosen with accordance to \eqref{nonsingular} and \eqref{strictneg} and $r>\max\{V(x)\colon |x|\<R\}$. Let $g\colon C(I,\R{N})\to\R{N}$ be a continuous function, which maps bounded sets into bounded sets and satisfies additionally the following conditions:
\begin{itemize}
\item[(i)] $\forall\, x\in AC(I,\R{N})\;\;\;[x(0)=g(x)]\Rightarrow\exists\,t\in(0,T]\;\;\;|g(x)|\<|x(t)|$,
\item[(ii)] $|g(i(x))|\<|x|\;$ for all $x\in V^{-1}(\{r\})$,
\item[(iii)] $\fix(g\circ i)\cap V^{-1}(\{r\})=\emptyset$.
\end{itemize}
Then the solution set $S_{\!F}(g)$ of nonlocal initial value problem \eqref{nonlocal} is nonempty and compact as a subset of $\,C(I,\R{N})$.
\end{theorem}
\begin{proof}
That the set $S_{\!F}(g)$ is not empty results from Theorem \ref{th4}. Just as in the proof of Theorem \ref{th6}. it suffices to show the boundedness of the solution set $S_{\!F}(g)$ to be certain of its compactness.\par Take $x\in S_{\!F}(g)$ and suppose that $|x(0)|>e^{||\mu||_1}(R+1)-1$. From Lemma \ref{lem2}. it follows that $|x(t)|>R$ for every $t\in I$. Assume that $|x(0)|\<|x(t_0)|$ for some $t_0\in(0,T]$. Then $V(x(t_0))-V(x(0))\geqslant 0$, due to the monotonicity of $V$. On the other hand $V(x(t_0))-V(x(0))=\int_0^{t_0}\langle\nabla V(x(s)),\dot{x}(s)\rangle\,ds<0$, because \[\langle\nabla V(x(s)),\dot{x}(s)\rangle\<\langle\nabla V(x(s)),F(s,x(s))\rangle^+<0\;\;\mbox{ for a.a. }s\in[0,t_0].\] We arrive at contradiction. If we assume the opposite, i.e. $|x(0)|>|x(t)|$ for every $t\in(0,T]$, then condition (i) is contradicted. Therefore it is impossible that the supposition $|x(0)|>e^{||\mu||_1}(R+1)-1$ was correct.\par The growth condition $(F_3)$ along with the Gronwall inequality implies \[|x(t)|\<\left[||\mu||_1+(R+1)e^{||\mu||_1}-1\right]e^{\int_0^t\mu(s)\,ds}\;\;\mbox{ for }t\in I.\] Thus $S_{\!F}(g)\subset\overline{B}(M)$, where
\[M:=\left[||\mu||_1+(R+1)e^{||\mu||_1}-1\right]e^{||\mu||_1}.\]
\end{proof}


\begin{thebibliography}{99}
\bibitem{aubin} J. Aubin, A. Cellina, {\it Differential Inclusions}, Springer, Berlin, 1984.
\bibitem{blasi} F. S. de Blasi, L. G\'orniewicz, G. Pianigiani, {\it Topological degree and periodic solutions of differential inclusions}, Nonlinear
Analysis 37 (1999), 217-245.
\bibitem{deimling} K. Deimling, {\it Multivalued differential equations}, Walter de Gruyter, Berlin-New York, 1992.
\bibitem{maw} R. Gaines, J. Mawhin, {\it Coincidence Degree and Nonlinear Differential Equations}, Lecture Notes in Mathematics, 568, Springer-Verlag, Berlin, 1977. 
\bibitem{gorn2} L. G\'orniewicz, {\it Topological fixed point theory of multivalued mappings}, Second ed., Springer, Dordrecht, 2006.
\bibitem{gorn} L. G\'orniewicz, S. Plaskacz, {\it Periodic solutions of differential inclusions in $\R{n}$}, Boll. Un. Mat. Ital. 7(7-A) (1993), 409-420.
\bibitem{zabreiko} M. A. Krasnosel'skii, P. P. Zabreiko, {\it Geometrical Methods of Nonlinear Analysis}, Springer-Verlag, Berlin, 1984.
\bibitem{obu} V. Obukhovskii, P. Zecca, N. van Loi, S. Kornev, {\it Method of Guiding Functions in Problems of Nonlinear Analysis}, Lecture Notes in Mathematics, 2076, Springer-Verlag, Berlin, 2013.
\end{thebibliography}
\end{document}